\documentclass[a4paper, 12pt]{amsart}

\usepackage[T1]{fontenc}
\usepackage[utf8]{inputenc}

\usepackage{amssymb}
\usepackage{mathrsfs}
\usepackage{mathtools}
\usepackage{enumitem}
\usepackage{stmaryrd}
\usepackage{hyperref}
\usepackage{microtype}
\usepackage{xparse}
\usepackage{cite}
\usepackage{scalerel}
\usepackage{graphicx}

\allowdisplaybreaks[2]

\newtheorem{theorem}{Theorem}[section]
\newtheorem{proposition}[theorem]{Proposition}
\newtheorem{lemma}[theorem]{Lemma}
\newtheorem{corollary}[theorem]{Corollary}
\theoremstyle{definition}
\newtheorem{remark}[theorem]{Remark}
\newtheorem{definition}[theorem]{Definition}

\numberwithin{equation}{section}

\newtheorem{mainthm}{Theorem}

\newcommand{\N}{\mathbb N}

\newcommand{\Z}{\mathbb Z}

\newcommand{\sa}{a}
\newcommand{\sac}{\overline{a}}

\newcommand{\sbb}{b}
\newcommand{\sbc}{\overline{b}}

\newcommand{\sbwi}{\widetilde{b}^{(1)}}
\newcommand{\sbwii}{\widetilde{b}^{(2)}}

\newcommand{\st}{\overline{t}}
\newcommand{\stc}{t}

\newcommand{\stt}{\overline{\alphaaa}}
\newcommand{\sttc}{\alphaaa}

\newcommand{\stwi}{\alphaaaw^{(1)}}
\newcommand{\stwii}{\alphaaaw^{(2)}}

\newcommand{\aaa}{\textbf{a}}
\newcommand{\aaaw}{\widetilde{\aaa}}
\newcommand{\aaac}{\overline{\textbf{a}}}
\newcommand{\aaact}{\aaac^\perp}
\newcommand{\aaat}{\aaa^\perp}
\newcommand{\bbb}{\textbf{b}}
\newcommand{\bbbw}{\widetilde{\bbb}}
\newcommand{\bbbt}{\bbb^\perp}
\newcommand{\ccc}{\textbf{c}}
\newcommand{\ccct}{\ccc^\perp}
\newcommand{\geng}{\textbf{g}}
\newcommand{\gengt}{\geng^\perp}
\newcommand{\sss}{\textbf{s}}

\newcommand{\alphaaa}{\textbf{t}}
\newcommand{\alphaaaw}{\widetilde{\alphaaa}}

\newcommand{\III}{\mathfrak{I}}
\newcommand{\JJJ}{\mathfrak{L}}
\newcommand{\AAA}{\mathfrak{A}}
\newcommand{\BBB}{\mathfrak{B}}

\newcommand{\TTT}{T}
\newcommand{\TTTw}{\widetilde{\TTT}}

\newcommand{\FFFw}{\widetilde{H}}

\newcommand{\MMM}{M}

\newcommand{\Laurent}[3]{#1( \! ( #2; #3 )\!)}

\DeclareMathOperator{\diag}{diag}

\DeclareMathOperator{\K}{\scalerel*{\kappa}{C}}
\DeclareMathOperator{\gld}{gld}
\DeclareMathOperator{\glr}{glr}
\DeclareMathOperator{\sr}{sr}

\DeclareMathOperator{\const}{const}
\DeclareMathOperator{\lowest}{lowest}

\ExplSyntaxOn
\NewDocumentCommand{\cycle}{ O{\;} m }
{
	(
	\alec_cycle:nn { #1 } { #2 }
	)
}

\seq_new:N \l_alec_cycle_seq
\cs_new_protected:Npn \alec_cycle:nn #1 #2
{
	\seq_set_split:Nnn \l_alec_cycle_seq { , } { #2 }
	\seq_use:Nn \l_alec_cycle_seq { #1 }
}
\ExplSyntaxOff

\title{Skew Laurent Series Ring Over a Dedekind Domain}

\author{Daniel Vitas}
\address{Department of Mathematics, Faculty of Mathematics and Physics,  University of Ljubljana, Slovenia \newline \indent 
Institute of Mathematics, Physics and Mechanics, Ljubljana, Slovenia}

\email{daniel.vitas@imfm.si}

\thanks{\emph{Mathematics Subject Classification} (2020). 16E60, 16N60, 16P40, 19A49}
\keywords{Skew Laurent series rings, noncommutative Dedekind domains, ideal class groups}

\begin{document}

\begin{abstract} We show that the formal skew Laurent series ring $R = \Laurent{D}{x}{\sigma}$ over a commutative Dedekind domain $D$ with an automorphism $\sigma$ is a noncommutative Dedekind domain. If $\sigma$ acts trivially on the ideal class group of $D$, then $K_0(R)$, the Grothendieck group of $R$, is isomorphic to $K_0(D)$. Furthermore, we determine the Krull dimension, the global dimension, the general linear rank, and the stable rank of $R$.
\end{abstract}

\maketitle

\section{Introduction}

Let $D$ be a commutative ring with an automorphism $\sigma\colon D \rightarrow D$ and let $R = \Laurent{D}{x}{\sigma}$
be the formal skew Laurent series ring over $D$, i.e., the ring that consists of the series $\sum_{i=k}^\infty a_i x^i$ with $k \in \Z$ and $a_i \in D$, and is subject to the rule $x a = \sigma(a) x$ for all $a \in D$. Hilbert used this construction in 1899 to provide the first known example of a centrally infinite division ring. This led several authors, including Dickson, Hahn, Schur, Mal'cev, and Neumann, to study this and similar rings. In \cite[Section 14]{Lam}, Lam gives a detailed account of the early development of the skew Laurent series ring and similar constructions.


It is well known that if $D$ is noetherian (resp., an integral domain, a field), then $R$ is also (left and right) noetherian (resp., a domain, a division ring) (see \cite[Example 1.8]{Lam} and \cite{Tu}). Several authors have studied extensions of various properties; Tuganbaev shows that $R$ is right serial right artinian if and only if $D$ is \cite{Tu4}, Romaniv and Sagan show that $R$ is $\omega$-Euclidean if so is $D$ \cite{RS},  
Mazurek and Ziembowski determine when $R$ satisfies an ascending chain condition on principal right ideals \cite{MZ} (while in \cite{MZ2} they also characterize generalized power series rings that are semilocal right Bézout; see also \cite{Sal1, Sal2}), Letzter and Wang determine the Goldie rank of $R$ \cite{LW}, Majidinya and Moussavi study the Baer property of $R$ \cite{Ma1, Ma2}, and Annin describes the associated primes of $R$ \cite{Ann}. However, there are very few established results in the case where $D$ is a Dedekind domain; in the commutative setting (i.e., if $\sigma = {\rm id}$), \cite[Ch 10.3, Ex 3]{Jac} shows that $R$ is also a Dedekind domain (see \cite{PO} for the generalized power series case). 

A similar problem in the noncommutative setting is studied in \cite[Section 7.9]{McR}. It is established that the skew polynomial ring $S = D[x,x^{-1}; \sigma]$, the subring of $R$ consisting of all the series with only finitely many nonzero terms, is a noncommutative Dedekind domain if $D$ is a Dedekind domain.
Furthermore, there is a surjective map
$$ \phi \colon K_0(D) \rightarrow K_0(S)$$
with the kernel generated by $[M^\sigma] - [ M ]$ for all finitely generated projective $D$-modules $M$. Here, $K_0(S)$ denotes the Grothendieck group of $S$, $[ M ]$ denotes the isomorphism class of $M$, and $M^\sigma$ denotes the $\sigma$-skewed module we get from $M$ (see Section \ref{sec1} and \cite[Section 12.5]{McR} for details). Also, note that the Grothendieck group of $D$ splits as
$$ K_0(D) = G(D) \oplus \Z \text{,}$$
where $G(D)$ is the ideal class group of $D$ \cite[Section 35]{LR}.

We say that $\sigma$ \emph{acts trivially on} $G(D)$ if $\sigma(\III)$ is isomorphic to $\III$ for any ideal $\III \lhd D$.  In this paper, we will prove the following theorem. 

\begin{mainthm}\label{mainthm}
	If $D$ is a commutative Dedekind domain with an automorphism $\sigma$, then the skew Laurent series ring $R = \Laurent{D}{x}{\sigma}$ is a noncommutative Dedekind domain. Furthermore, if $\sigma$ acts trivially on $G(D)$, then $$ K_0(R) \cong K_0(D) \text{.}$$
\end{mainthm}

We will actually show something more about the ideal structure of $R$, namely, we will prove that every right ideal $I_R \leq R_R$ is isomorphic to the extended ideal $\III R$ for some $\III \lhd D$ (Proposition \ref{extension lemma}); yet not every right ideal of $R$ is of this form (e.g., the ideal $(2+x)R$ for $D = \Z$). With this we also get the following theorem.

\begin{mainthm}\label{mainthm2}
Let $D$ be a commutative Dedekind domain with an automorphism $\sigma$ that acts trivially on $G(D)$, and let $R = \Laurent{D}{x}{\sigma}$. Then any two stably isomorphic finitely generated projective $R$-modules are isomorphic.
\end{mainthm}

Note that this is not in general true for all noncommutative Dedekind domains; the first Weyl algebra $A_1(k)$ is a simple noncommutative Dedekind domain, but has stably free modules that are not free \cite[Corollary 11.2.11]{McR}. In the setting of classical maximal orders in central simple algebras
over number fields the situation is subtle: stable isomorphism implies isomorphism unless the algebra is a totally definite quaternion algebra. (This is a consequence of strong approximation.) In the exceptional case of totally definite quaternion algebras, on the other hand, there are only finitely many instances where stable isomorphism still implies isomorphism. These have all been classified \cite{Sm3}.

Finally, we also compute a few invariants of $R$; namely, its right Krull dimension $\K(R)$, right global dimension $\gld(R)$, general linear rank $\glr(R)$, and stable rank $\sr(R)$.

\begin{mainthm}
Let $D$ be a commutative Dedekind domain, that is not a field, with an automorphism $\sigma$, and let $R = \Laurent{D}{x}{\sigma}$. Then $\K(R) = 1$, $\gld(R) = 1$, and $\sr(R) = 2$. If $\sigma$ acts trivially on $G(D)$, then $\glr(R)=1$.
\end{mainthm}

\section{Preliminaries} \label{sec1}

Let $D$ be a commutative domain and $K$ its field of fractions. We say that $D$ is a \emph{Dedekind domain} if every nonzero ideal $\III \lhd D$ is invertible, i.e., there exists a finitely generated $D$-module $\JJJ \subseteq K$ (we call such modules \emph{fractional ideals}) such that $\III \JJJ = D$. Throughout this paper, $D$ will always denote a commutative Dedekind domain. Let $\sigma \colon D \rightarrow D$ be an automorphism of $D$. The skew Laurent series ring over $D$ is the set
$$ R = \Laurent{D}{x}{\sigma}= \Big\{ \,\sum_{i = k}^\infty a_i x^i \,\,  \Big\vert \,\, a_i \in D, \,\, k \in \Z \, \Big\}$$
with addition defined term-wise and multiplication defined by using the distributive property and the formula
$$ x^i a = \sigma^i (a) x$$
for any $a \in D$ and $i \in \Z$.
We will often write $Q$ for $\Laurent{K}{x}{\sigma}$ (note that this is not the ring of fractions of $R$, but it is a division ring containing $R$).
%

For a Laurent series
\begin{align} \label{right-form}
	f = \sum_{i=k}^\infty a_i x^i \in R
\end{align}
with $a_k \neq 0$ we define the \emph{lowest coefficient of $f$} to be
$ \lowest(f) = a_k$ and write $\lowest(f) = 0$ if $f = 0$.
For a right ideal $I_R \leq R_R$ the set
$$  \const(I) = \left\{ \, \lowest(f)  \,\, \middle\vert \,\, f \in I \, \right\} $$
is an ideal of $D$; we call it \emph{the constant ideal of $I$}. The ideal $\const(I)$ plays an important role in studying the ideal $I$.

\begin{remark}
	We will establish our first result only for right ideals. To invoke the left-right symmetry and use the same result for left ideals, we have to adjust our terms a little bit. We can write every element of $R$ as in \eqref{right-form} in the \emph{left normal form} as
	$$ f = \sum_{i=k}^\infty x^i b_i$$
	for $b_i = \sigma^{-i}(a_i)$. It is easy to check that elements in the left normal form multiply symmetrically to those in $R$ (except applying $\sigma^{-1}$ instead of $\sigma$ to the elements in $D$). By defining the \emph{\textup{(}left\textup{)} lowest coefficient of $f$} to be $\lowest_l(f) = \sigma^{-k}(a_k)$ and writing
	$$\const_l(J) =  \left\{ \, \lowest_l(f)  \,\, \middle\vert \,\, f \in I  \, \right\}$$
	for a left ideal $\prescript{}{R}{J} \leq \prescript{}{R}{R}$, we obtain a symmetric notion to $\const(I)$ defined for a right ideal $I_R \leq R_R$. In this context, applying all the steps in future proofs symmetrically from the left will yield a symmetric result for the left ideals, despite the fact that the proofs are only written for the right ideals.
\end{remark}

The following lemma is found in \cite[Lemma 4.3]{Tu} and will be important later on.
In particular, it already shows that if $D$ is a PID, then $R$ is also a PID.

\begin{lemma} \label{gen lemma}
	Let $I_R \leq R_R$ be a right ideal and let 
	$$ \const(I) = a_1 D + \dots + a_n D$$
	for some $a_i \in D$. Then there exist $f_1, \ldots, f_n \in R$ with $\lowest(f_i) = a_i$ such that
	$$ I = f_1 R + \dots + f_n R \text{.}$$
\end{lemma}

Let us recall a few notions. Let $S$ be an arbitrary (unital) ring. We say that $S$ is a \emph{noncommutative Dedekind domain} if it is a domain, it is noetherian (i.e., both left and right ideals are finitely generated), it is hereditary (i.e., both left and right ideals are projective), and it is an Asano order (i.e., two sided nonzero ideals are both left and right invertible); see \cite[Definition 23.5]{LR}. (The left hereditary property follows from the right hereditary property, as long as the ring $S$ is left and right noetherian  \cite[Corollary 7.65]{Lam2}.)
Note that for a commutative ring $S$, this notion coincides with the standard notion of a commutative Dedekind domain.

This type of rings admits factorization properties of modules, similar to those of commutative Dedekind domains (see \cite[Section 5.7]{McR} or \cite[Sections 33, 35]{LR}). In particular, every finitely generated module is a direct sum of a projective module and a torsion module. Furthermore, every finitely generated projective module $M$ is isomorphic to
\begin{align} \label{ideal property}
	M \simeq I \oplus S^n
\end{align}
for some right ideal $I_S \leq S_S$ and $n \in \N_0 = \N \cup \{ 0\}$.
Noncommutative Dedekind domains also admit a cancellation property. More precisely, if
$$ I \oplus S^n \simeq J \oplus S^n$$
for some right ideals $I$ and $J$ and $n \in \N$, then
\begin{align} \label{canc property}
	I \oplus S \simeq J \oplus S \text{.}
\end{align}
See \cite[Corollary 11.7.14]{McR} for details.

For any ring $S$, not necessarily a Dedekind domain, we say that two right $S$-modules $M_S$ and $N_S$ are \emph{stably isomorphic} if
$$ M \oplus S^n \simeq N \oplus S^n$$
for some $n \in \N_0$.
In general, the Grothendieck group $K_0(S)$ of any ring $S$ is a group generated by stable isomorphism classes of finitely generated projective (right) modules $M_S$ (we denote the stable isomorphism class of $M$ by $[M]$), where the addition is given by
$$ [ M ] +[ N ] = [ M \oplus N ] \text{.}$$
For a noncommutative Dedekind domain $S$, we can define the homomorphism
$$\psi \colon K_0(S) \rightarrow \Z \text{,}$$
mapping an element of the form $[I \oplus S^n] - [ J \oplus S^m]$ for some nonzero ideals $I$ and $J$ to $n-m$. The kernel of $\psi$ consists of elements of the form $[I]-[S]$ (this follows from \cite[Theorem 11.7.13(ii)]{McR}), where $I$ is a right ideal of $S$; we denote it by
$$ G(S) = \ker \psi$$
and call it \emph{the (right) ideal class group} of $S$. Since $\Z$ is projective, the morphism $\psi$ splits and we have
$$ K_0(S) \cong G(S) \oplus \Z \text{.}$$
Note that if $S$ is a commutative ring, then $G(S)$ coincides with the usual ideal class group. See \cite[Section 12.1]{McR} for details.

\section{Extension Proposition}

Recall that $D$ is a commutative Dedekind domain with an automorphism $\sigma \colon D \rightarrow D$ and $R = \Laurent{D}{x}{\sigma}$ is the skew Laurent series ring. This section is dedicated to proving the following proposition.

\begin{proposition} \label{extension lemma}
	Let $I_R \leq R_R$ be a right ideal. Then
	$$ I \simeq \III R$$
	as right $R$-modules, where $\III = \const(I) \lhd D$.
\end{proposition}

\begin{remark}
	In future, we will use the fact that
	$$ \III R = \Laurent{\III}{x}{\sigma} = \Big\{ \, \sum_{i = k}^\infty a_i x^i \,\,  \Big\vert \,\, a_i \in \III, \,\, k \in \Z \, \Big\}$$
	without any further reference. To check that this is true, note that the ideal $\III \lhd D$ is generated by two elements, say $g_1$ and $g_2$ (in fact, since $D$ is a commutative Dedekind domain, every ideal has the $1 \frac{1}{2}$-generator property: for every nonzero $a \in \III$ there is an element $b \in \III$ such that $a$ and $b$ generate $\III$). The left-to-right inclusion is obvious, so let us take an
	$$ a = \sum_{i = k}^\infty a_i x^i$$
	with $a_i \in \III$, and show that $a \in \III R$. Since $g_1$ and $g_2$ generate $\III$, we can write each $a_i$ as
	$$ a_i = g_1 a_{i1} + g_2 a_{i2}$$
	for some $a_{i1}, a_{i2} \in D$. Thus,
	$$ a = g_1 \left( \sum_{i = k}^\infty a_{i1} x^i \right) + g_2 \left(  \sum_{i = k}^\infty a_{i2} x^i \right)\text{,}$$
	which is now obviously an element of $\III R$.
\end{remark}

From now on let $I_R \leq R_R$ be a right ideal, and write $\III$ for $\const(I)$. Since $D$ is a Dedekind domain, every ideal in $D$ is generated by two elements; let $a_1,a_2 \in D$ generate $\III$. Set
$$\geng_0 = \begin{bmatrix} a_1 & a_2 \end{bmatrix} \quad \text{and} \quad \gengt_0 = \begin{bmatrix} a_2 \\ -a_1 \end{bmatrix}\text{.}$$
By Lemma \ref{gen lemma}, there are $f_1, f_2 \in I$ with $\lowest(f_i) = a_i$ that generate $I$. Without loss of generality we may assume that
$$ f_i = a_i + \sum_{j=1}^\infty a_{ij} x^j$$
for some $a_{ij} \in D$.
Define
$$ \geng_j = \begin{bmatrix} a_{1j} & a_{2j} \end{bmatrix}$$
for every $j \in \N$ and set
$$\geng  = \sum_{j=0}^\infty \geng_j x^j  = \begin{bmatrix} f_1 & f_2 \end{bmatrix}\text{.}$$

Since the entries of $\geng$ generate $I$, and the entries of $\geng_0$ generate $\III$, we have the following relation on the rows $\geng_i$.
	
\begin{lemma} \label{lemma s_i}
	For every $n \in \N_0$ and $$\sss_0, \sss_1, \dots, \sss_{n-1} \in \begin{bmatrix} D \\ D \end{bmatrix}$$ such that
	\begin{align} \label{predp0} \sum_{i=0}^k \geng_i \sigma^i ( \sss_{k-i} ) = 0 \end{align}
	for every $k = 0, 1, \ldots, n-1$, we have
	$$ \sum_{i=1}^n \geng_i \sigma^i ( \sss_{n-i} ) \in \III \text{.}$$
\end{lemma}

\begin{proof}
Set
$ \sss = \sum_{i=0}^{n-1} \sss_{i} x^i\text{.}$
Since the entries of $\geng$ generate $I$, we have
$$ \geng \sss \in I \text{.}$$
Computing this expression yields
\begin{align*}
	\geng \sss &= \left( \sum_{i=0}^\infty \geng_i x^i \right) \left( \sum_{i=0}^{n-1} \sss_i x^i \right) \\
	&= \sum_{k=0}^{n-1} \sum_{i=0}^k \geng_i \sigma^i ( \sss_{k-i} )  x^k + \sum_{k=n}^{\infty} \sum_{i=k-n+1}^k \geng_i \sigma^i ( \sss_{k-i} )  x^k \text{.}
\end{align*}
By assumption \eqref{predp0}, the first summand is zero, so we have that
$$ \geng \sss = \sum_{i=1}^n \geng_i \sigma^i ( \sss_{n-i} )  x^n +  h x^{n+1} \in I$$
for some $h = \sum_{i=0}^\infty h_i x^i$ with $h_i \in D$. By definition of the constant ideal, this means that
$$\sum_{i=1}^n \geng_i \sigma^i ( \sss_{n-i} )  \in \III \text{.}  \eqno \qedhere$$
\end{proof}

The following lemma gives us sufficient conditions on $\geng_i$ for the ideal $I$ to be isomorphic to $\III R$. 

\begin{lemma} \label{lemma A_i}
	If there exist matrices $A_i \in M_2(D)$, with $A_0$ being the identity matrix, such that
	\begin{align} \label{predp} \sum_{i=0}^k \geng_i  \sigma^i \left( A_{k-i} \right)\sigma^k(\gengt_0) = 0\end{align}
	holds for all $k \in \N_0$, then
	$ I \simeq \III R$.
\end{lemma}

\begin{proof}
Recall that we write $K$ for the field of fractions of $D$ and $Q$ for $\Laurent{K}{x}{\sigma}$. Set
$$ A =  \sum_{i=0}^\infty A_i x^i \in M_2(R)$$
and let
$$q = 1 + \sum_{i=1}^\infty q_i x^i \in Q \text{,}$$
where $q_i \in K$ have yet to be determined.
We will define $q_i$ so that we have
\begin{align} \label{q} q \geng_0 = \geng A \text{.}\end{align}
Since the matrix $A$ is invertible \cite[Proposition 4.4]{Tu}, the two entries of $\geng A$ generate $I$. Therefore \eqref{q} implies that $\III R$ is isomorphic to $I$ (via multiplication by $q$). The equation \eqref{q} is equivalent to the system
\begin{align*} q_k \sigma^k (\geng_0) = \sum_{i=0}^k \geng_i \sigma^i (A_{k-i}) \end{align*}
for $k \in \N$. Since these are $2$-dimensional rows and $\sigma^k (\geng_0)$ is nonzero, the existence of such $q_i \in K$ is equivalent to the condition \eqref{predp}.
\end{proof}

Now all that is left is to construct the matrices $A_i$ from the previous lemma.

\begin{proof}[Proof of Proposition \ref{extension lemma}]
We assume that $\III \neq 0$, since otherwise $I = 0$ as well and the statement is clear. We will prove the existence of matrices $A_i$ from Lemma \ref{lemma A_i} by induction. The equality \eqref{predp} for $k=0$ is trivial (since $A_0$ is the identity matrix), so assume that $n \geq 1$ and that there exist matrices $A_1$,~\ldots,~$A_{n-1}$ such that \eqref{predp} holds for every $k=0,1,\ldots, n-1$. Since $D$ is a Dedekind domain and $\III$ is a nonzero ideal of $D$, there is a fractional ideal $\III^{-1} \subseteq K$ (i.e., a finitely generated $D$-module) such that
\begin{align} \label{inverse}
    \III \III^{-1} = D \text{.}
\end{align}
We will first show that
\begin{align}\label{ideal cont}
\sum_{i=1}^n \geng_i \sigma^i (A_{n-i}) \sigma^n ( \gengt_0) \sigma^n (\III^{-1}) \subseteq \III \text{.}
\end{align}
To prove this take a $q \in \III^{-1}$ and let
$$\sss_i = A_i \sigma^i (\gengt_0) \sigma^i (q) \in \begin{bmatrix} D \\ D \end{bmatrix}$$
for $i=0,1,\ldots,n-1$. Then \eqref{predp} implies \eqref{predp0}, and by Lemma \ref{lemma s_i}, we have
$$ \sum_{i=1}^n \geng_i \sigma^i ( \sss_{n-i} ) =  \sum_{i=1}^n \geng_i \sigma^i (A_{n-i}) \sigma^n ( \gengt_0) \sigma^n (q) \in \III \text{.}$$
Since this holds for every $q \in \III^{-1}$, inclusion \eqref{ideal cont} follows.

By multiplying inclusion \eqref{ideal cont} with the ideal $\sigma^n(\III)$, we get
$$\sum_{i=1}^n \geng_i \sigma^i (A_{n-i}) \sigma^n ( \gengt_0) \sigma^n (\III^{-1} \III) \subseteq \III \sigma^n (\III) \text{.}$$
By \eqref{inverse}, we have $\sigma^n(\III^{-1} \III) = D$, which implies
$$ \sum_{i=1}^n \geng_i \sigma^i (A_{n-i}) \sigma^n ( \gengt_0) \in \III \, \sigma^n ( \III ) \text{.}$$
Since the entries of $\geng_0$ and $\sigma^n (\gengt_0)$ generate $\III$ and $\sigma^n (\III)$, respectively, this means that there exists a matrix $A_n \in M_2(D)$ such that
$$ -\geng_0 A_n \sigma^n ( \gengt_0)  = \sum_{i=1}^n \geng_i \sigma^i (A_{n-i}) \sigma^n ( \gengt_0 ) \text{.}$$
This concludes the induction step. Applying Lemma \ref{lemma A_i} concludes the proof.
\end{proof}

\section{Laurent series ring is a Dedekind domain}

We are now in position to establish one of the main results. 

\begin{theorem} \label{dedekind}
	Let $D$ be a commutative Dedekind domain with an automorphism $\sigma$. Then $R = \Laurent{D}{x}{\sigma}$ is a noncommutative Dedekind domain. Furthermore, the ring $R$ is simple if and only if $\sigma(\III) \neq \III$ for every nonzero proper ideal $\III \lhd D$.
\end{theorem}

\begin{proof}
	The fact that $R$ is a domain is a standard result.
	%
	The ring $R$ is noetherian due to Proposition \ref{extension lemma}.
	To see that $R$ is hereditary, first we prove that it is flat as a left $D$-module. Indeed, direct products of flat modules are flat over a noetherian ring $D$, and $R$ can be written as a union of
	$$ R_k = \Big\{ \, \sum_{i = k}^\infty a_i x^i \,\,  \Big\vert \,\, a_i \in D \, \Big\} \simeq \prod_{i=k}^{\infty} D$$
	as left $D$-modules. In fact, $R$ is a direct limit of $R_k$ with the natural embeddings and is thus itself flat.
    Now take any right ideal $I_R \leq R_R$ and let $\III = \const(I)$.
	Proposition \ref{extension lemma} shows that $I \simeq \III R$. Since $R$ is a flat $D$-module, we have that $\III R \simeq \III \otimes_D R$. The projective property is preserved by tensoring with flat modules, hence $I$ is projective. This shows that $R$ is right hereditary; the left hereditary property follows by symmetry (or, alternatively, it follows from  \cite[Corollary 7.65]{Lam2}).
	
	Now, we will prove that $R$ is an Asano order. Take a nonzero two-sided ideal $I \lhd R$. Since $I$ is (in particular) a right ideal, Proposition \ref{extension lemma} implies that there is a $q \in Q$ and an ideal $\III \lhd D$ such that
	$ I = q \III R$.
	Consider the left fractional ideal
	$J = R \III^{-1} q^{-1} R \text{.}$
	Since $R I = I$, we see that
	\begin{align} \label{idealeq1}
		J I = R \text{.}
	\end{align}
	Similarly, since $I$ is also a left ideal, the left symmetric version of Proposition \ref{extension lemma} implies that there is a $p \in Q$ and an ideal $\JJJ \lhd D$ such that $I = R \JJJ p$. For a right fractional ideal $J' = R p^{-1} \JJJ^{-1} R$, we have
	\begin{align} \label{idealeq2}
		I J' = R \text{.}
	\end{align}
	Multiplying equation \eqref{idealeq1} from the right by $J'$ and considering \eqref{idealeq2}, yields $J = J'$. Therefore $J$ is a two sided fractional ideal with $IJ = JI=R$.
	
	Finally, if $I \lhd R$ is a two-sided ideal of $R$, then $x I = I$, which implies that $\sigma(\III) = \III$ for the ideal $\III = \const(I) \lhd D$. So if there are no such nonzero proper ideals, then either $I = 0$ or $I = R$. Conversely, if 
$\III \lhd D$ is a nonzero proper ideal with $\sigma(\III) = \III$, then $I = \III R$ is a nonzero proper two-sided ideal of $R$.
\end{proof}

For the ring $R$, we will now determine its Krull dimension $\K(R)$, right global dimension $\gld(R)$, and stable rank $\sr(R)$. For definitions  see \cite[Sections 6.2.2, 7.1.8, 11.3.4]{McR}.

\begin{theorem}
    Let $D$ and $R$ be as in the previous theorem, and assume that $D$ is not a field. Then $\K(R) = 1$, $\gld(R) = 1$, and $\sr(R) = 2$.
\end{theorem}

\begin{proof}
    By Theorem \ref{dedekind}, $R$ is hereditary, noetherian, and prime, and thus by \cite[Corollary 6.2.8]{McR}, we have $\K(R) \leq 1$. If $\K(R) = 0$, then $R$ would be (right) artinian, hence $D$ would be an artinian Dedekind domain, and thus a field. Therefore $\K(R) = 1$.

    By Theorem \ref{dedekind}, $R$ is (right) hereditary, and thus $\gld(R) \leq 1$. Since $R$ is not (right) artinian, as before, we have $\gld(R) = 1$.

    By \cite[Corollary 6.7.4]{McR}, we have that $\sr(R) \leq \K(R) +1 = 2$. We will prove that $\sr(R) \geq 2$ by finding a two dimensional unimodular row $\aaa$ that is not stable, i.e.,
    $$ \aaa = \begin{bmatrix} a_1 & a_2\end{bmatrix} \in \begin{bmatrix} R & R \end{bmatrix}$$
    such that
    $$ a_1 R + a_2 R = R \text{,}$$
    but $a_1 + a_2 r$ is not invertible for any $r \in R$.
    To find such an $\aaa$, let $a \in D$ be a nonzero element that is not invertible. For $b = 1 + \sigma(a)$, the row $$ \aaa = \begin{bmatrix} a + x & a^2 + bx\end{bmatrix}$$
    has the wanted properties.
    Indeed, $\aaa$ is unimodular, since
    $$ (a+x)a - (a^2+bx) = (\sigma(a) - b) x = - x$$
    is an invertible element of $R$.
    We will prove that $\aaa$ is not stable. For contradiction, assume that there is an $r \in R$ such that
    \begin{align*}
    	f = (a+x) + (a^2+bx) r \in R
    \end{align*}
    is invertible. Clearly $r \neq 0$, so write $r = \sum_{i=k}^\infty r_i x^i$ for some $r_i \in D$ and $k \in \Z$ with $r_k \neq 0$. The lowest coefficient
    $$ \lowest(f) = \begin{cases} a^2 r_k & k < 0 \\ a + a^2 r_0 & k = 0 \\ a & k > 0 \end{cases}$$
    cannot be invertible or zero (since neither is $a$), but this contradicts the invertibility of $f$.
\end{proof}

We have yet to compute the general linear rank of $R$, defined in \cite[Section 11.1.14]{McR}. In the following section we will show that $\glr(R) = 1$.

\section{Computing the ideal class group}

Since, by Theorem \ref{dedekind}, $R$ is a Dedekind domain, we have that $K_0 (R) \cong G(R) \oplus \Z$, where $G(R)$ denotes the (right) ideal class group of $R$. Therefore, determining $K_0(R)$ is the same as determining $G(R)$. We can verify that $G(R)$ is isomorphic to the group of stable isomorphism classes $[I]$ of (right) ideals $I_R \leq R_R$, where the addition is given by
$$ [I] + [J] = [K] \quad \text{if and only if} \quad [I \oplus J] = [K \oplus R] \text{.}$$

Proposition \ref{extension lemma} shows that the natural map $\phi \colon G(D) \rightarrow G(R)$, mapping $[ \III ]$ to $[\III R]$, is surjective.
%
The kernel of $\phi$ obviously contains all the elements of the form
\begin{align}\label{ker} [ \sigma(\III) ] - [ \III ] \text{,}\end{align}
since $\sigma(\III)R \simeq \III R$ via multiplication by $x^{-1}$.
Note that in a commutative Dedekind domain $D$ stable isomorphism classes are exactly the same as isomorphism classes, while in a noncommutative Dedekind domain $R$ the former can be strictly larger than the latter.

Throughout this section assume that $\sigma$ acts trivially on the ideal class group of $D$, i.e., $\sigma(\III)$ is isomorphic to $\III$ for any ideal $\III \lhd D$. This is equivalent to $\JJJ^{-1} \sigma(\JJJ)$ being a principal fractional ideal for any nonzero fractional ideal $\JJJ$.
In this case, the elements in \eqref{ker} vanish. The following theorem, that we will prove at the end of this section, states that the kernel of $\phi$ is then trivial.

\begin{theorem} \label{G theorem}
	Let $D$ be a commutative Dedekind domain with an automorphism $\sigma$ that acts trivially on $G(D)$, and let $R = \Laurent{D}{x}{\sigma}$. Then $\phi \colon G(D) \rightarrow G(R)$, sending $[\III]$ to $[\III R]$, is an isomorphism. In that case,
	$$ K_0(R) \cong K_0(D) \text{.}$$
\end{theorem}

Let us also point out a corollary to this theorem, which shows that in $R$ stable isomorphism yields an isomorphism.

\begin{corollary}
	Let $D$ and $R$ be as in Theorem \ref{G theorem}. Then any two stably isomorphic finitely generated projective $R$-modules are isomorphic.
\end{corollary}

\begin{proof}
	Take $M$ and $N$ to be stably isomorphic finitely generated projective right $R$-modules. If either $M$ or $N$ are zero, the statement trivially follows, so assume that they are both nonzero. Since $R$ is a noncommutative Dedekind domain, by \eqref{ideal property}, there are nonzero (right) ideals $I_R, J_R \leq R_R$ such that
	\begin{align} \label{ex formula}
		M \simeq I \oplus R^n \quad \text{and} \quad N \simeq J \oplus R^m
	\end{align}
	for some $n, m \in \N_0$. By computing the uniform dimension, since $N$ and $M$ are stably isomorphic, we see that $n = m$. By Proposition \ref{extension lemma}, there are ideals $\III, \JJJ \lhd D$ such that
	\begin{align} \label{repr formula}
		I \simeq \III R \quad \text{and} \quad J \simeq \JJJ R \text{.}
	\end{align}
	Since $M$ and $N$ are stably isomorphic, in particular, $\III R$ and $\JJJ R$ are also stably isomorphic, which implies that $[ \III ] - [ \JJJ ]$ lies in the kernel of $\phi$. Since $\phi$ is injective by Theorem \ref{G theorem}, we have $\III \simeq \JJJ$. The modules $M$ and $N$ are then clearly isomorphic by \eqref{ex formula} and \eqref{repr formula}.
\end{proof}

In particular, the above corollary states that every finitely generated stably free $R$-module is actually free. The following observation follows from \cite[Proposition 11.1.12]{McR}.

\begin{corollary}
    Let $D$ and $R$ be as in Theorem \ref{G theorem}. Then $\glr(R) = 1$.
\end{corollary}

The rest of this section is dedicated to the proof of Theorem \ref{G theorem}. We will require a few auxiliary definitions and lemmas. Let $I_R \leq R_R$ be a nonzero right ideal of $R$. By $I^{-1}$ denote the set
$$ I^{-1} = (R :_{l} I) = \{ \, q \in Q \mid q I \subseteq R  \, \} \text{.}$$
For example, if we consider the ideal $I = \III R$ for some ideal $\III \lhd D$, we have
$$ I^{-1} = \Big\{ \, \sum_{i=k}^{\infty} \sigma^i(q_i) x^i \,\,  \Big\vert \,\, q_i \in \III^{-1} , \,\, k \in \Z \, \Big\} \text{.}$$

Let $J$ be another right ideal of $R$. Essentially by \cite[Proposition 3.1.15]{McR}, we have that ${\rm Hom} (I, J) \cong J I^{-1}$ and we can identify ${\rm Hom}(R \oplus I, R \oplus J)$ with $2 \times 2$ matrices with entries in
$$\begin{bmatrix} R & I^{-1} \\ J & J I^{-1} \end{bmatrix} \text{.}$$
In particular,
$${\rm End}(R \oplus I) \cong \begin{bmatrix} R & I^{-1} \\ I & I I^{-1} \end{bmatrix} \text{,}$$
where the right-hand-side ring is a unital ring as well.

We will say that a row $\aaa \in \begin{bmatrix} R & I^{-1} \end{bmatrix}$ is \emph{invertible} if there is another row
$\bbb \in  \begin{bmatrix} I & I I^{-1} \end{bmatrix}$ such that the matrix
$$ A = \begin{bmatrix} \aaa \\ \bbb  \end{bmatrix}  \in \begin{bmatrix} R & I^{-1} \\ I & I I^{-1} \end{bmatrix}$$
is invertible (in this ring).

\begin{lemma} \label{irrelevant lemma} Let $I$ be a nonzero right ideal of $R$. A row $\aaa \in \begin{bmatrix} R & I^{-1} \end{bmatrix}$ is invertible if and only if there is an invertible matrix $$T \in \begin{bmatrix} R & I^{-1} \\ I & I I^{-1} \end{bmatrix}$$
such that
$ \aaa T = \begin{bmatrix} 1 &0\end{bmatrix}$.
\end{lemma}

\begin{proof}
	Assume that $\aaa$ is invertible and let $A$ be the matrix from the definition before the lemma. Set $T = A^{-1}$ and note that $T$ is an invertible matrix. Since $\aaa$ is the first row of the matrix $A$, we have
	$ \aaa T = \begin{bmatrix} 1 &0\end{bmatrix}$
	as required.
	
	Now assume that there is an invertible matrix $$T \in \begin{bmatrix} R & I^{-1} \\ I & I I^{-1} \end{bmatrix}$$ with
	$ \aaa T = \begin{bmatrix} 1 &0\end{bmatrix}$.
	Multiplying this equation by $T^{-1}$, we see that $\aaa$ is the first row of $T^{-1}$, which proves that $\aaa$ is invertible.
\end{proof}

We will say that a row $\aaa \in \begin{bmatrix} R & I^{-1} \end{bmatrix}$ is \emph{unimodular} if there is a column
$$ \alphaaa \in  \begin{bmatrix} R \\ I \end{bmatrix}$$
with
$$ \aaa  \alphaaa  = 1 \text{.}$$

The following lemma provides a sufficient condition for a right ideal $I_R \leq R_R$ to have the property that any stable isomorphism yields an isomorphism (see \cite[Proposition 11.1.12]{McR} for the case where $I =R$).

\begin{lemma} \label{stab imp iso}
Let $I$ be a nonzero right ideal of $R$ such that every unimodular row $\aaa \in \begin{bmatrix} R & I^{-1} \end{bmatrix}$ is invertible.
Then, if $R \oplus I \simeq R \oplus J$ for some other right ideal $J_R \leq R_R$, we must have $I \simeq J$.
\end{lemma}

\begin{proof}
	Assume that $R \oplus I \simeq R \oplus J$. By computing the uniform dimension, we see that $J$ is nonzero. Therefore, by the identification with $2 \times 2$ matrices mentioned before, there exist matrices
	$$ A \in \begin{bmatrix} R & I^{-1} \\ J & J I^{-1} \end{bmatrix} \quad \text{and} \quad B \in \begin{bmatrix} R & J^{-1} \\  I & I J^{-1} \end{bmatrix} \text{,}$$
	which are inverse to each other, i.e., $AB = BA = I_{2\times 2}$. Denote by $\aaa$ the first row of $A$ and by $\alphaaa$ the first column of $B$. Since $\aaa \alphaaa = 1$, the row $\aaa$ is unimodular, and thus, by assumption, it is invertible. By Lemma \ref{irrelevant lemma}, there is an invertible matrix
	$$ T \in \begin{bmatrix} R & I^{-1} \\ I & I I^{-1} \end{bmatrix}$$
	such that $ \aaa T = \begin{bmatrix} 1 &0\end{bmatrix}$. Then the matrix
	$A' = A T$ 
	is a lower triangular matrix with the inverse
	$B' = T^{-1} B$, which is also lower triangular, since $R$ is a domain. 
	In particular, for the bottom right entry of $A'$, say $a' \in J I^{-1}$, and the bottom right entry of $B'$, say $b' \in I J^{-1}$, we have $a' b'=b' a' =1$, which implies $I \simeq J$.
\end{proof}

To prove Theorem \ref{G theorem} we will prove that any unimodular row $\aaa \in \begin{bmatrix} R & I^{-1} \end{bmatrix}$ is invertible. By Proposition \ref{extension lemma}, it is enough to consider ideals of the form $I = \III R$ for some nonzero $\III \lhd D$. We have $\aaa = \sum_{i=k}^{\infty} \aaa_i x^i$ for some
$$\aaa_i \in  \begin{bmatrix} D & \sigma^{i}(\III^{-1})\end{bmatrix} \text{.}$$
Without loss of generality, we may assume that $k = 0$ and $\aaa_0 \neq 0$. Indeed, a row $\aaa$ is unimodular if and only if $x^{-k}\aaa$ is. To see this, let $$\alphaaa \in \begin{bmatrix} R \\ I \end{bmatrix}$$ be such that $\aaa \alphaaa = 1$. Then $$\alphaaa x^k \in \begin{bmatrix} R \\ I \end{bmatrix}$$ and we have $\left( x^{-k}\aaa  \right) \left(  \alphaaa x^{k} \right) = 1$, proving that $x^{-k} \aaa$ is unimodular. The reverse implication follows by symmetry. Similarly,  a row $\aaa$ is invertible if and only if $x^{-k}\aaa$ is. Thus, we can really assume that $k=0$ and $\aaa_0 \neq 0$.

We start with a very general definition. It might not be clear at first why this definition is needed (namely, why is the ideal $\AAA$ appearing in the definition), but it turns out to be a crucial generalization for proving the wanted theorem.

\begin{definition} \label{def1}
	Let $n \in \N_0$, and let $\AAA \lhd D$ be a nonzero ideal. For each  $i \in \N_0$, let $\aaa_i \in \begin{bmatrix} \sigma^{i}(\AAA^{-1})& \sigma^{i}(\III^{-1}\AAA)\end{bmatrix}$. We say that $\aaa = \sum_{i=0}^\infty \aaa_i x^i$ is \emph{$n$-unimodular} with respect to $\AAA$ if and only if there is a column $\alphaaa = \sum_{i=0}^{\infty} \alphaaa_i x^i$ with $$\alphaaa_i \in \begin{bmatrix} \AAA \\ \III \AAA^{-1}\end{bmatrix}$$ such that 
	$$ \aaa \alphaaa = x^n + \sum_{i=n+1}^\infty h_i x^{i}$$
	for some $h_i \in D$.
\end{definition}

\begin{lemma} \label{lem1}
	If a row $\aaa = \sum_{i=0}^{\infty} \aaa_i x^i \in \begin{bmatrix} R & I^{-1}\end{bmatrix}$ 
	is unimodular, it is $n$-unimodular with respect to $\AAA = D$ for some $n \in \N_0$.
\end{lemma}

\begin{proof}
	Assume that $\aaa$ is unimodular, i.e., there is a vector $$\alphaaa = \sum_{i=-n}^\infty \alphaaa_i x^i \in \begin{bmatrix} R \\ I\end{bmatrix}$$ such that $\aaa \alphaaa = 1$. Let $\alphaaa' = \alphaaa x^{n} = \sum_{i=0}^\infty \alphaaa_{i-n} x^i$. Since $$\alphaaa_{i-n} \in \begin{bmatrix} D \\ \III \end{bmatrix}$$ for every $i$ and $$\aaa \alphaaa' = x^n\text{,}$$ this shows that $\aaa$ is $n$-unimodular with respect to $\AAA = D$.
	%
\end{proof}

We can define a generalization of invertibility of a row in a similar way.

\begin{definition} \label{def}
	Let $n \in \N_0$, and let $\AAA \lhd D$ be a nonzero ideal. For each $i \in \N_0$, let $\aaa_i \in \begin{bmatrix} \sigma^{i}(\AAA^{-1})& \sigma^{i}(\III^{-1}\AAA)\end{bmatrix}$. We say that $\aaa = \sum_{i=0}^\infty \aaa_i x^i$ is \emph{$n$-invertible} with respect to $\AAA$ if and only if there is a row $\bbb = \sum_{i=0}^\infty \bbb_i x^i$ with $\bbb_i \in \begin{bmatrix} \III\sigma^{i}(\AAA^{-1})& \III\sigma^{i}(\III^{-1}\AAA)\end{bmatrix}$ and a matrix
	$T  = \sum_{i=0}^{\infty} T_i x^i$
	with $$T_i \in \begin{bmatrix} \AAA& \sigma^{i}(\III^{-1}) \AAA \\ \III \AAA^{-1} & \III \sigma^{i}(\III^{-1})\AAA^{-1}\end{bmatrix}$$ such that 
	$$ \begin{bmatrix} \aaa \\ \bbb \end{bmatrix} T = \sum_{i=n}^{\infty} H_i x^{i}$$
	for some matrices $$H_i \in \begin{bmatrix} D & \sigma^{i}(\III^{-1}) \\ \III & \III \sigma^i (\III^{-1}) \end{bmatrix}$$ with $\det(H_n)$ generating $\III \sigma^n(\III^{-1})$.
\end{definition}

\begin{lemma} \label{lem2}
	If a row $\aaa = \sum_{i=0}^{\infty} \aaa_i x^i \in \begin{bmatrix} D & I^{-1}\end{bmatrix}$ 
	is $n$-invertible with respect to $\AAA = D$ for some $n \in \N_0$, then it is invertible.
\end{lemma}

\begin{proof}
%
	Assume that $\aaa$ is $n$-invertible with respect to $\AAA = D$.
    Take $\bbb$, $T$ and $H_i$ to be as in the previous definition.
    %
    Let $\mu = 1/\det(H_n)$, denote $M = \diag(1, \mu)$, and let
	$$ S = \begin{bmatrix} \aaa \\ \bbb \end{bmatrix} T x^{-n} M = H_n M + \sum_{i=1}^\infty H_{i+n} \sigma^i(M) x^i \text{.}$$
	Since $$H_n M \in \begin{bmatrix} D & \III^{-1} \\ \III & \III \III^{-1} \end{bmatrix}$$ has determinant $1$, it is an invertible matrix in this ring, and thus, the matrix $S$ is invertible in $$\begin{bmatrix} D &I^{-1} \\ I & I I^{-1} \end{bmatrix} \text{.}$$ Therefore, $T' = T x^{-n} M S^{-1}$ is an inverse of the matrix $$A = \begin{bmatrix} \aaa \\ \bbb \end{bmatrix}\text{,}$$ which shows that $\aaa$ is invertible.
\end{proof}

In our next step we will, for a row $\aaa$, define the row $\aaaw$. We will show that if $\aaa$ is $n$-unimodular, then $\aaaw$ is $n-1$-unimodular, and that if $\aaaw$ is $n-1$-invertible, then $\aaa$ is $n$-invertible (with respect to a certain ideal). This will enable us to prove that $R$ satisfies the assumption of Lemma \ref{stab imp iso} by a simple induction on $n$. But before that, we need to decompose rows in a suitable way.

For a row $\ccc = \begin{bmatrix} c^{(1)} & c^{(2)} \end{bmatrix} \in \begin{bmatrix} K & K \end{bmatrix}$, denote by $\ccct$ the orthogonal column
$$ \ccct = \begin{bmatrix} c^{(2)} \\ -c^{(1)}  \end{bmatrix} \text{,}$$
which satisfies $ \ccc \ccct = 0$.
Let $\aaa = \sum_{i=0}^{\infty} \aaa_i x^i$ with $\aaa_i \in \begin{bmatrix} \sigma^{i}(\AAA^{-1})& \sigma^{i}(\III^{-1}\AAA)\end{bmatrix}$ for $i \in \N_0$ and $\aaa_0 \neq 0$. Then $\aaa_0 = \begin{bmatrix} s & q \end{bmatrix}$ for some $s \in \AAA^{-1}$ and $q \in \III^{-1} \AAA$. Let
$$\BBB = s \AAA + q \III \AAA^{-1} \lhd D \text{.}$$
Since $s \AAA \BBB^{-1} + q \III \AAA^{-1} \BBB^{-1} = D$, there is a row $\aaac_0 \in \begin{bmatrix} \III \AAA^{-1} \BBB^{-1} & \AAA \BBB^{-1} \end{bmatrix}$ such that
$$ \det \begin{bmatrix} \aaa_0 \\ \aaac_0 \end{bmatrix} = \aaa_0 \aaact_0 = 1 \text{.}$$
Fix such an $\aaac_0$. Note that the matrix
\begin{align} \label{A} A = \aaact_0 \aaa_0 - \aaat_0 \aaac_0 \end{align}
satisfies $\aaa_0 A = \aaa_0$ and $\aaac_0 A = \aaac_0$, and since $\aaa_0$ and $\aaac_0$ are linearly independent, this implies that $A$ is the identity matrix. This enables us to decompose rows in a suitable manner, namely, we can write rows $\aaa_i \in \begin{bmatrix} \sigma^{i}(\AAA^{-1})& \sigma^{i}(\III^{-1}\AAA)\end{bmatrix}$ as
$$ \aaa_i = \aaa_i \sigma^i (A) = \sa_i \sigma^i(\aaa_0) - \sac_i \sigma^i(\aaac_0)$$
with $\sa_i = \aaa_i \sigma^i(\aaact_0)\in \sigma^i(\BBB^{-1})$ and $\sac_i = \aaa_i \sigma^i(\aaat_0)\in \sigma^i(\III^{-1} \BBB)$.

Fix a $\lambda \in K$ such that
\begin{align}\label{lambda}  \lambda D =  \III \BBB^{-1} \sigma(\III^{-1} \BBB)\end{align}
and denote
$$\aaaw_i = \begin{bmatrix} \sa_i & \sac_{i+1} / \sigma^{i}(\lambda) \end{bmatrix} \in \begin{bmatrix} \sigma^i(\BBB^{-1}) & \sigma^i (\III^{-1} \BBB) \end{bmatrix} \text{.}$$
Let $\aaaw = \sum_{i=0}^\infty \aaaw_i x^i$. The next two lemmas connect the previously defined properties of a row $\aaa$ to the row $\aaaw$.

\begin{lemma} \label{lemma uni}
	Let $n \geq 1$. If a row $\aaa$ is $n$-unimodular with respect to $\AAA$, then $\aaaw$ is $n-1$-unimodular with respect to $\BBB$.
\end{lemma}

\begin{proof}
    Assume that $\aaa$ is $n$-unimodular with respect to $\AAA$, i.e., there is a column $\alphaaa = \sum_{i=0}^{\infty} \alphaaa_i x^i$ with $$\alphaaa_i \in \begin{bmatrix} \AAA \\ \III \AAA^{-1}\end{bmatrix}$$ such that 
	$$ \aaa \alphaaa = x^n + \sum_{i=n+1}^{\infty} h_i  x^i$$
	for some $h_i \in D$.
	This is equivalent to the system of equations
	\begin{align} \label{system1}
		\sum_{i=0}^k \aaa_i \sigma^i(\alphaaa_{k-i}) = \delta_{kn}
	\end{align}
	for $k = 0, 1, \ldots, n$. Here, $\delta_{kn} = 1$ if $k = n$, and $\delta_{kn} = 0$ otherwise.
	For the matrix $A$ from \eqref{A}, we have
	$$\alphaaa_i = A \alphaaa_i =   \aaact_0 \st_i + \aaat_0 \stc_i$$
	for $\st_i = \aaa_0 \alphaaa_i \in \BBB$ and $\stc_i= -\aaac_0 \alphaaa_i \in \III \BBB^{-1}$. 	Then, the system of equations \eqref{system1} can be written as
	\begin{align} \label{system2}
	 \sum_{i=0}^k \sa_i \sigma^i(\st_{k-i}) + \sac_i \sigma^i(\stc_{k-i}) = \delta_{kn}
	 \end{align}
	 for $k = 0, 1, \ldots, n$.
	In particular, taking into account that $\sa_0 = 1$ and $\sac_0 = 0$, the $k=0$ equation of \eqref{system2} states
	$$ \st_0 = 0 \text{.}$$
	Considering this, we can rewrite the system \eqref{system2} as
	\begin{align} \label{system3}
	 \sum_{i=0}^{k-1} \sa_i \sigma^i(\st_{k-i}) + \sac_{i+1} \sigma^{i+1}(\stc_{k-i-1}) = \delta_{kn}
	 \end{align}
	 for $k=1,\ldots,n$. For $\lambda$ from \eqref{lambda}, let 
	 $$ \alphaaaw_i = \begin{bmatrix} \st_{i+1} \\ \lambda \, \sigma(\stc_{i}) \end{bmatrix} \in \begin{bmatrix} \BBB \\ \III \BBB^{-1} \end{bmatrix} \text{.}$$
	 Then, we can write the system of equations \eqref{system3} as
	 \begin{align*}
	 \sum_{i=0}^{k} \aaaw_i \sigma^{i}( \alphaaaw_{k-1})= \delta_{k,n-1}
	 \end{align*}
	 for $k = 0, 1, \ldots, n-1$, which shows that $\aaaw$ is $n-1$-unimodular with respect to $\BBB$.
\end{proof}


\begin{lemma} \label{lemma inv}
	Let $n \geq 1$. If $\aaaw$ is $n-1$-invertible with respect to $\BBB$, then $\aaa$ is $n$-invertible with respect to $\AAA$.
\end{lemma}

\begin{proof}
	Assume that $\aaaw$ is $n-1$-invertible with respect to $\BBB$, i.e., there is a row $\bbbw = \sum_{i=0}^\infty \bbbw_i x^i$ with $\bbbw_i \in \begin{bmatrix} \III\sigma^{i}(\BBB^{-1})& \III\sigma^{i}(\III^{-1}\BBB)\end{bmatrix}$ and a matrix
	$\TTTw  = \sum_{i=0}^{\infty} \TTTw_i x^i$
	with $$\TTTw_i \in \begin{bmatrix} \BBB& \sigma^{i}(\III^{-1}) \BBB \\ \III \BBB^{-1} & \III \sigma^{i}(\III^{-1})\BBB^{-1}\end{bmatrix}$$ such that 
	$$ \begin{bmatrix} \aaaw \\ \bbbw \end{bmatrix} \TTTw = \sum_{i = n -1}^\infty \FFFw_i x^i$$
	where $$\FFFw_i \in \begin{bmatrix} D & \sigma^{i}(\III^{-1}) \\ \III & \III \sigma^{i} (\III^{-1}) \end{bmatrix}$$ with $\det(\FFFw_{n-1})$ generating $\III \sigma^{n-1}(\III^{-1})$.
	This is equivalent to the system of equations
	\begin{align} \label{sys1}
	\sum_{i=0}^k \begin{bmatrix} \aaaw_i \\ \bbbw_i \end{bmatrix} \sigma^i(\TTTw_{k-i}) = \FFFw_{n-1} \delta_{k,n-1}
	\end{align}
	for $k = 0,1,\ldots,n-1$. Write $\sbwi_i$ and $\sbwii_i$ for the entries of $\bbbw_i$.
	For $\lambda$ from \eqref{lambda}, let
	$$\sbb_i = \sbwi_i \quad \text{and} \quad \sbc_i = \sbwii_{i-1} \sigma^{i-1}(\lambda)$$
	for $i \in \N_0$ with $\sbc_0 = 0$,
	and set
	$$\bbb_i = \sbb_i \sigma^i(\aaa_0) - \sbc_i \sigma^i(\aaac_0) \in \begin{bmatrix} \III\sigma^{i}(\AAA^{-1})& \III\sigma^{i}(\III^{-1}\AAA)\end{bmatrix} \text{.}$$
	The row $\sum_{i=0}^\infty \bbb_i x^i$ will turn out to be the row $\bbb$ from Definition \ref{def}. We have yet to construct the matrix $\TTT$ from the definition.
	For this purpose, let $\mu$ be a generator for $\III \sigma(\III^{-1})$ and denote $\MMM = \diag(1, \mu)$.
	Write $\stwi_i$ and $\stwii_i$ for the rows of $\TTTw_i$.
	Let
	$$\stt_i = \stwi_{i-1} \sigma^{i-1}(M) \quad \text{and} \quad \sttc_i = \sigma^{-1}(\stwii_i/\lambda) \sigma^{i-1}(M)$$
	for $i \in \N_0$ with $ \stt_0 = 0$, and set
	$$ \TTT_i = \aaact_0 \stt_i + \aaat_0 \sttc_i \in \begin{bmatrix} \AAA& \sigma^{i}(\III^{-1}) \AAA \\ \III \AAA^{-1} & \III \sigma^{i}(\III^{-1})\AAA^{-1}\end{bmatrix} \text{.}$$
	For $k  = 0,1, \ldots,n-1$ we can compute
	\begin{align} \label{sys2}
	\begin{split}
	\sum_{i=0}^{k+1} \begin{bmatrix} \aaa_i \\ \bbb_i \end{bmatrix} \sigma^i(\TTT_{k+1-i}) &= \sum_{i=0}^{k+1} \begin{bmatrix} \sa_i  \sigma^i(\stt_{k+1-i}) + \sac_i \sigma^i(\sttc_{k+1-i}) \\ \sbb_i  \sigma^i(\stt_{k+1-i}) + \sbc_i  \sigma^i(\sttc_{k+1-i})  \end{bmatrix} \\
	&= \sum_{i=0}^{k} \begin{bmatrix} \sa_i  \sigma^i(\stt_{k+1-i}) + \sac_{i+1}  \sigma^{i+1}(\sttc_{k-i}) \\ \sbb_i  \sigma^i(\stt_{k+1-i}) + \sbc_{i+1}  \sigma^{i+1}(\sttc_{k-i})  \end{bmatrix} \\
	&=  \sum_{i=0}^{k} \begin{bmatrix} \aaaw_i \\ \bbbw_i \end{bmatrix} \sigma^i \left( \begin{bmatrix} \stt_{k+1-i} \\  \lambda \, \sigma(\sttc_{k-i}) \ \end{bmatrix} \right) \text{,}
	\end{split}
	\end{align}
	where in the second line we used the fact that $\sac_0 = \sbc_0 = 0$ and $\stt_0 = 0$.
	Since we have
	$$ \begin{bmatrix} \stt_{k+1-i} \\  \lambda \, \sigma(\sttc_{k-i}) \ \end{bmatrix} = \TTTw_{k-i} \sigma^{k-i}(\MMM)\text{,}$$
	the system of equations \eqref{sys2} can be written as
	$$\sum_{i=0}^{k+1} \begin{bmatrix} \aaa_i \\ \bbb_i \end{bmatrix} \sigma^i(\TTT_{k+1-i}) = \left( \sum_{i=0}^{k} \begin{bmatrix} \aaaw_i \\ \bbbw_i \end{bmatrix} \sigma^i(\TTTw_{k-i})\right) \sigma^k (\MMM) = \FFFw_{n-1} \sigma^k(M) \delta_{k,n-1}$$
	for $k = 0, 1, \ldots, n-1$. Let
	$$H_n = \FFFw_{n-1} \sigma^{n-1}(M) \in  \begin{bmatrix} D & \sigma^{n}(\III^{-1}) \\ \III & \III \sigma^{n} (\III^{-1}) \end{bmatrix}$$
	and note that $\det (H_n)$ generates $\III \sigma^{n}(\III^{-1})$. Using
	$$ \begin{bmatrix} \aaa_0 \\ \bbb_0 \end{bmatrix} \TTT_{0} = 0 \text{,}$$
	we see that
	$$ \begin{bmatrix} \aaa \\ \sum_{i=0}^{\infty} \bbb_i x^i \end{bmatrix} \sum_{i=0}^{\infty}  \TTT_{i} x^i = H_n x^n + \sum_{i=n+1}^\infty H_i x^i$$
	for some matrices
	$$H_i \in \begin{bmatrix} D & \sigma^{i}(\III^{-1}) \\ \III & \III \sigma^{i} (\III^{-1}) \end{bmatrix} \text{,}$$
	which proves that $\aaa$ is $n$-invertible with respect to $\AAA$.
\end{proof}

This final lemma before the proof of the main theorem essentially shows that the ring $R$ satisfies the assumption of Lemma \ref{stab imp iso}. 

\begin{lemma} \label{induction}
If $\aaa$ is $n$-unimodular with respect to $\AAA$, then it is $n$-invertible with respect to $\AAA$.
\end{lemma}

\begin{proof}
We proceed by induction on $n$, where $\AAA$ varies over all nonzero ideals of $D$ and $\aaa$ over all suitable rows. First consider the $n=0$ case.
Assume that $\aaa$ is $0$-unimodular with respect to $\AAA$, i.e., there is a column $\alphaaa = \sum_{i=0}^{\infty} \alphaaa_i x^i$ with $$\alphaaa_i \in \begin{bmatrix} \AAA \\ \III \AAA^{-1}\end{bmatrix}$$ such that 
	$$ \aaa \alphaaa = 1 + \sum_{i=1}^{\infty} h_i x^i$$
	for some $h_i \in D$.
	Take any nonzero $s \in \III$ and $b \in \AAA$. Then $s b \in \AAA$. Since $D$ is a Dedekind domain, there is an $a \in \AAA$ such that $a$ and $sb$ generate $\AAA$ (this is the $1 \frac{1}{2}$-generator property). Then there exist $p, q \in \AAA^{-1}$ such that
	$$ a p + sbq = 1 \text{.}$$
	Thus, the matrix
	$$ \TTT_0 = \begin{bmatrix} a & b \\ -sq & p \end{bmatrix} \in \begin{bmatrix} \AAA& \III^{-1} \AAA \\ \III \AAA^{-1} & \AAA^{-1}\end{bmatrix}$$
	has determinant $1$. Let $\bbb_0 \in \begin{bmatrix} \III\AAA^{-1}& \AAA \end{bmatrix}$ be a row such that
	$$ \alphaaa_0 = \bbbt_0 \text{.}$$
	 Taking $\TTT = \TTT_0$ and $\bbb = \bbb_0$, we see that
	 $$ \begin{bmatrix} \aaa \\ \bbb \end{bmatrix} \TTT = \begin{bmatrix} \aaa_0 \\ \bbb_0 \end{bmatrix} \TTT_0 + \sum_{i=1}^\infty H_i x^i$$
	 for some matrices
	 $$H_i \in \begin{bmatrix} D & \sigma^{i}(\III^{-1}) \\ \III & \III \sigma^{i} (\III^{-1}) \end{bmatrix} \text{,}$$
	 and since
	 $$\det \left( \begin{bmatrix} \aaa_0 \\ \bbb_0 \end{bmatrix} \TTT_0\right) = \aaa_0 \alphaaa_0 \det ( \TTT_0) = 1 \text{,}$$
	 this shows that $\aaa$ is $0$-invertible with respect to $\AAA$.
	 
	 Now let $n \geq 1$ and assume that the lemma holds for $n-1$ (and all nonzero ideals $\AAA$ and suitable rows $\aaa$). Let $\aaa$ be $n$-unimodular with respect to $\AAA$. By Lemma \ref{lemma uni}, $\aaaw$ is $n-1$-unimodular with respect to $\BBB$. By induction hypothesis, $\aaaw$ is $n-1$-invertible with respect to $\BBB$, and finally, by Lemma \ref{lemma inv}, $\aaa$ is $n$-invertible with respect to $\AAA$. This concludes the induction step.
\end{proof}

We are now in position to prove the main theorem.

\begin{proof}[Proof of Theorem \ref{G theorem}]
	Take any nonzero ideal $\III \lhd D$ and denote $I = \III R$. First, we will show that we can use Lemma \ref{stab imp iso}, i.e., we will show that any unimodular row $\aaa \in \begin{bmatrix} R & I^{-1} \end{bmatrix}$ is invertible. As pointed out in a paragraph before Definition \ref{def1}, we can assume that $\aaa = \sum_{i=0}^{\infty} \aaa_i x^i$ for some
$\aaa_i \in  \begin{bmatrix} D & \sigma^{i}(\III^{-1})\end{bmatrix}$ and $\aaa_0 \neq 0$. By Lemma \ref{lem1}, $\aaa$ is $n$-unimodular with respect to $\AAA = D$ for some $n \in \N_0$, by Lemma \ref{induction}, $\aaa$ is then $n$-invertible with respect to $\AAA = D$, and finally, by Lemma \ref{lem2}, $\aaa$ is invertible. This shows that the assumption of Lemma \ref{stab imp iso} is satisfied.

Take any other ideal $\JJJ \lhd D$ and denote $J = \JJJ R$. Assume that $I$ and $J$ are stably isomorphic in $R$.  Since $R$ is a noncommutative Dedekind domain, the cancellation property \eqref{canc property} implies that $R \oplus I \simeq R \oplus J$. Lemma \ref{stab imp iso} shows that $I \simeq J$, i.e., there is a $q = \sum_{i=k}^\infty q_i x^i \in Q$ with $q_k \neq 0$ such that
$$ q I = J \text{.}$$
In particular, this implies that $q_k \sigma^k(\III) = \JJJ$. Therefore, we have $\sigma^k(\III) \simeq \JJJ$, but since $\sigma$ acts trivially on $G(D)$, this shows that $\III \simeq \JJJ$. Therefore, the map $\phi \colon G(D) \rightarrow G(R)$ is injective, and thus, by Proposition \ref{extension lemma}, an isomorphism.
\end{proof}

\section*{Acknowledgment}

The author would like to thank his supervisor Daniel Smertnig for his guidance throughout this work. He would also like to thank the referee for their helpful comments.
The author was supported by the Slovenian Research and Innovation Agency (ARIS) program P1-0288.

\end{document}